\documentclass[12pt,reqno]{amsart}
\usepackage{amscd,amsmath,amsthm,amssymb}
\usepackage{color}
\usepackage{tikz}
\usepackage{url}
\usepackage{circuitikz}
\usepackage{float}

 \newtheorem{Theorem}{Theorem}[section]
 \newtheorem{Lemma}[Theorem]{Lemma}
 \newtheorem{Corollary}[Theorem]{Corollary}
 \newtheorem{Proposition}[Theorem]{Proposition}
 \newtheorem{Remark}[Theorem]{Remark}
 
 \newtheorem{Example}[Theorem]{Example}

 \newtheorem{Conjecture}{Conjecture}

 \def\qed{\ifhmode\textqed\fi
       \ifmmode\ifinner\quad\qedsymbol\else\dispqed\fi\fi}
 \def\textqed{\unskip\nobreak\penalty50
        \hskip2em\hbox{}\nobreak\hfill\qedsymbol
        \parfillskip=0pt \finalhyphendemerits=0}
 \def\dispqed{\rlap{\qquad\qedsymbol}}

\def\Ass{\textup{Ass}}

\def\reg{\textup{reg}}
 
\begin{document}

\title{Componentwise linear symbolic powers of edge ideals and Minh's conjecture}
\author{Antonino Ficarra, Somayeh Moradi, Tim Römer}

\address{Antonino Ficarra, Departamento de Matem\'{a}tica, Escola de Ci\^{e}ncias e Tecnologia, Centro de Investiga\c{c}\~{a}o, Matem\'{a}tica e Aplica\c{c}\~{o}es, Instituto de Investiga\c{c}\~{a}o e Forma\c{c}\~{a}o Avan\c{c}ada, Universidade de \'{E}vora, Rua Rom\~{a}o Ramalho, 59, P--7000--671 \'{E}vora, Portugal}
\email{antonino.ficarra@uevora.pt\,\,,\quad antficarra@unime.it}

\address{Somayeh Moradi, Department of Mathematics, Faculty of Science, Ilam University, P.O.Box 69315-516, Ilam, Iran}
\email{so.moradi@ilam.ac.ir}

\address{Tim Römer, Osnabrück University, Osnabrück, Germany}
\email{troemer@uos.de}


\begin{abstract}
In this paper, we study the componentwise linearity of symbolic powers of edge ideals. We propose the conjecture that all symbolic powers of the edge ideal of a cochordal graph are componentwise linear. This conjecture is verified for some families of cochordal graphs, including complements of block graphs and complements of proper interval graphs. As a corollary, Minh's conjecture is established for such families. Moreover, we show that $I(G)^{(2)}$ is componentwise linear, for any cochordal graph $G$.
\end{abstract}

\maketitle
\vspace*{-0.6cm}
\section*{Introduction}
Let $S=K[x_1,\ldots,x_n]$ be the standard graded polynomial ring over a field $K$. By a classical result of Cutkosky, Herzog and Trung \cite{CHT}, and independently Kodiyalam \cite{K}, the regularity of powers of a graded ideal $I\subset S$ is an eventually linear function. This had a great impact on the study of homological invariants of powers of graded ideals. A prominent trend in commutative algebra is to  explicitly determine this function for combinatorially defined monomial ideals. For instance, consider a finite simple graph $G$ with the vertex set $V(G)=\{x_1,\ldots,x_n\}$ and edge set $E(G)$. The \textit{edge ideal} of $G$ is the squarefree monomial ideal of the polynomial ring $S$ defined as $I(G)=(x_ix_j:\, \{x_i,x_j\}\in E(G))$. Then, there exist integers $k_0>0$ and $c\ge0$ such that $\reg\,I(G)^k=2k+c$ for all $k\ge k_0$. Determining the integers $k_0$ and $c$ in terms of combinatorics of $G$ is a problem of great interest. In recent years, the study of symbolic powers of monomial ideals, and in particular, edge ideals, has also gained significant attention, see for instance \cite{BC,DHNT,GHOS,JK,KKS,MNPTV,MTr,MVu,F1,F3,F2} and the references therein. While the regularity of symbolic powers of monomial ideals is known to be a quasi-linear function~\cite[Corollary 3.3]{HHT}, its precise behavior remains mysterious. In this context, N.C. Minh raised the following 
\begin{Conjecture}\label{ConjA}
    Let $I(G)$ be the edge ideal of a simple graph $G$. Then
    $$\reg\, I(G)^{(k)}=\reg\, I(G)^k,$$
    for all $k\ge1$.
\end{Conjecture}

If this conjecture holds, then $\reg\,I(G)^{(k)}$ would be also an eventually linear function, a result that would be both surprising and impactful. Thus far, Conjecture \ref{ConjA} has been proved when $k=2,3$ for any graph~\cite{MNPTV}, and for few families of graphs, including bipartite graphs~\cite{SVV}, chordal graphs~\cite{F2}, unicyclic graphs~\cite{F1} and Cameron-Walker graphs~\cite{F3}. Note that $\reg\,I(G)^{k}\ge 2k$ for all $k$. Therefore, if Conjecture \ref{ConjA} is true, then $\reg\,I(G)^{(k)}\ge2k$ for all $k$. This naturally leads to the question: for which graphs $G$ does the equality $\reg\,I(G)^{(k)}=2k$  hold for all $k$? In particular, in this case $I(G)$ itself must have linear resolution. By Fröberg's seminal work \cite{Froberg88}, this is equivalent to  $G$ being a cochordal graph, meaning that the complementary graph $G^c$ of $G$ is chordal. Moreover, an intriguing result of Herzog, Hibi and Zheng \cite{HHZ} establishes that if $I(G)$ has linear resolution, then $I(G)^k$ has linear resolution for all $k\ge1$. This is further equivalent to $I(G)^k$ having linear quotients for all $k\ge1$.

Now, let $I(G)$ be an edge ideal with linear resolution. In general, the symbolic powers $I(G)^{(k)}$ are not equigenerated, so one cannot expect that they have linear resolution, like in the case of the ordinary powers. However, based on several computational evidence, we expect that each graded component of $I(G)^{(k)}$ has linear resolution, i.e., $I(G)^{(k)}$ is \textit{componentwise linear}. Componentwise linear ideals were introduced in \cite{HH99} by Herzog and Hibi, as those homogeneous ideals $I\subset S$ whose all graded components $I_{\langle j\rangle}$ have linear resolution. Recall that the $j^{\textup{th}}$ graded component of $I$ is the ideal generated by all homogeneous elements of degree $j$ belonging to $I$. Componentwise linear ideals are characterized by the remarkable property that their graded Betti numbers are equal to those of their generic initial ideals~\cite{AHH}.

These aforementioned considerations on edge ideals of cochordal graphs led us to formulate the following
\begin{Conjecture}\label{ConjB}
    Let $I(G)$ be the edge ideal of a simple graph $G$. Assume that $I(G)$ has linear resolution. Then $I(G)^{(k)}$ is componentwise linear for all $k\ge1$.
\end{Conjecture}

It follows from Theorem \ref{Thm:perfect} in Section \ref{sec1} that for any cochordal graph $G$, the highest generating degree of $I(G)^{(k)}$ is $2k$. Since the highest generating degree of a componentwise linear ideal $I$ is equal to $\reg\,I$ \cite[Corollary 8.2.14]{HHBook}, if Conjecture \ref{ConjB} is true, then Conjecture \ref{ConjA} is true for any cochordal graph $G$.
Our main goal in this paper is to address Conjectures \ref{ConjA} and \ref{ConjB}.

In Section \ref{sec1}, using a description of the symbolic Rees algebra of the edge ideal of a perfect graph $G$ due to Villarreal \cite{Vil}, we determine in Theorem \ref{Thm:perfect} the generating degrees of $I(G)^{(m)}$ for any perfect graph $G$ and any positive integer $m$. In particular, it turns out that the highest degree of a minimal generator of $I(G)^{(m)}$ is $2m$, and if $G$ is cochordal then $\reg\,I(G)^{(m)}\ge\reg\,I(G)^{m}$ (see Corollary \ref{inequalityreg}). Furthermore, we obtain an explicit formula for the Waldschmidt constant of the edge ideal of a perfect graph, recovering a result of Bocci {\it et al.} \cite[Theorem 6.7(i)]{BC}.

In Section \ref{sec2}, Conjecture \ref{ConjB} is proved for the graphs whose complements are one of the following graphs:
\begin{enumerate}
\item[(a)] Block graphs,
\item[(b)] Proper interval graphs,
\item[(c)] $G$ is a chordal graph with the property that any vertex in $G$ belongs to at most two maximal cliques of $G$.
\end{enumerate}

Indeed, we prove a more general result. Let $G$ be one of the graphs in (a), (b) or (c), and let $\Ass\,I(G)=\{P_1,\dots,P_m\}$ be the set of associated primes of $I(G)$. Then $I(G)^{(k)}=\bigcap_{i=1}^mP_i^k$.   Theorem \ref{Thm:forestComLinear} shows that $\bigcap_{i=1}^mP_i^{k_i}$ is componentwise linear for any positive integers $k_1,\ldots,k_m$. In the literature, such ideals are called intersection of Veronese ideals, and their componentwise linearity was first studied by Francisco and Van Tuyl \cite{FVT}. Another family of such ideals with the componentwise linear property was given in \cite[Theorem 2.4]{MM}. Theorem \ref{Thm:forestComLinear} presents several new such families. We expect that for any cochordal graph $G$, the ideal $\bigcap_{P\in\Ass\,I(G)}P^{k_P}$ is componentwise linear for any positive integers $k_P$.

A monomial ideal $I\subset S$ has \textit{componentwise linear quotients} if $I_{\langle j\rangle}$ has linear quotients for all $j$. These ideals are componentwise linear. To prove Theorem \ref{Thm:forestComLinear}, we show that $I(G)^{(k)}$ has componentwise linear quotients. A conjecture by Soleyman Jahan and Zheng \cite{SJZ} states that if $I$ has componentwise linear quotients, then $I$ has linear quotients. This conjecture is widely open, and has been solved only in some special cases \cite{BQ1,BQ2,Fic0}. Given these facts, along with the proof of Theorem \ref{Thm:forestComLinear}, we expect that the following more general statement than Conjecture \ref{ConjB} holds true.

\begin{Conjecture}\label{ConjC}
    Let $I(G)$ be the edge ideal of a simple graph $G$. Assume that $I(G)$ has linear resolution. Then $I(G)^{(k)}$ has linear quotients for all $k\ge1$.
\end{Conjecture}

In Theorem \ref{Thm:I2}, we show that Conjecture \ref{ConjC} holds for $k=2$. The proof is based on Theorem \ref{Thm:perfect}(a) and some splittings of the $t$-clique ideals of $G$.

\section{The generating degrees of symbolic powers of edge ideals of perfect graphs}\label{sec1}

In this section, the generating degrees of $I(G)^{(m)}$ are studied, when $G$ is a perfect graph. As a result, we derive an inequality in Conjecture~\ref{ConjA} for cochordal graphs. We begin the discussion with  some definitions and notation.
Throughout, $G$ is a finite simple graph. The vertex set and the edge set of $G$ are denoted by $V(G)$ and $E(G)$, respectively.

A graph $G$ is called {\em chordal} if it has no induced cycles of length $r>3$, and $G$ is called {\em cochordal}, if the complementary graph $G^c$ of $G$ is chordal. Here, $G^c$ is the graph with the same vertex set as $G$ whose edges are the non-edges of $G$.
A graph $G$ is called a {\em perfect} graph, if $G$ and $G^c$ do not contain induced odd cycles of length $r>3$. The family of perfect graphs contains for instance the families of bipartite graphs, weakly chordal graphs (and in particular chordal graphs and cochordal graphs) and comparability graphs of  posets. 

For a subset $A\subseteq V(G)$,  the induced subgraph of $G$ on $A$ is denoted by $G[A]$. A {\em clique} of $G$ is a subset $C\subseteq V(G)$ such that the induced subgraph $G[C]$ is a complete graph. A clique of size $r$ is called an {\em $r$-clique}. The maximum cardinality of the cliques of $G$ is denoted by $\omega(G)$ and is called the {\em clique number} of $G$. 

Villarreal in~\cite[Corollary 3.3]{Vil} gave a description for the symbolic Rees algebra $$\mathcal{R}_s(I(G))=S\oplus I(G)^{(1)}t\oplus \cdots \oplus I(G)^{(i)}t^i\oplus\cdots\subseteq S[t]$$ of $I(G)$, when $G$ is a perfect graph in terms of the cliques of $G$, as follows:
\begin{equation}\label{terriblesossage}
\mathcal{R}_s(I(G))=K[\textbf{x}_Ft^r: F \textrm{ is an $(r+1)$-clique of $G$}],
\end{equation}
where $\textbf{x}_F=\prod_{x_i\in F}x_i$. 

For a positive integer $r$, the \textit{$r$-clique ideal} of $G$ was defined in~\cite{Mor} as
$$
K_r(G)=(\textbf{x}_F: F \textrm{ is an $r$-clique of $G$}).
$$

When $\omega(G)=1$, we have $I(G)=(0)$. So excluding this case, in the following theorem we assume that $\omega(G)\geq 2$.  

\begin{Theorem}\label{Thm:perfect}
    Let $G$ be a perfect graph with  the clique number $\omega=\omega(G)\geq 2$. Then, for all $m\ge1$,
\begin{enumerate}
\item[(a)] $
    I(G)^{(m)}=\sum K_{s_1}(G)K_{s_2}(G)\cdots K_{s_j}(G)$, where the sum is taken over all integers $1\le j\le m$ and all integers $s_1,\dots,s_j$ such that $2\le s_i\le \omega$ for all $i$, and $s_1+\dots+s_j=m+j$.
\item[(b)] $\beta_{0,d}(I(G)^{(m)})\neq 0$ if and only if $d=m+j$ with $\lceil m/(\omega-1)\rceil\le j\le m$.
\end{enumerate}
\end{Theorem}

\begin{proof}
(a)  It follows from equation (\ref{terriblesossage}) that $$I(G)^{(m)}t^m=\sum K_{s_1}(G)K_{s_2}(G)\cdots K_{s_j}(G)t^{(s_1-1)+\cdots +(s_j-1)},$$
where the   sum is taken over  all integers $s_1,\dots,s_j\geq 2$ for some $j$ such that $(s_1-1)+\cdots +(s_j-1)=m$.  This is equivalent to 
$s_1+\dots+s_j=m+j$ and  $2\le s_i\le \omega$ for all $i$, since $K_{s_i}(G)=(0)$ for $s_i>\omega$. Moreover, from the inequalities $s_i\geq 2$, we obtain $2j\leq \sum_{i=1}^j s_i=m+j$ and hence
$j\le m$.

\medskip
(b) First we prove the `if' statement.  Let $q=\lceil m/(\omega-1)\rceil$, and $q\le j\le m$ be an integer. We need to show that $I(G)^{(m)}$ has a minimal generator of degree $m+j$.

First we claim that there exist integers $s_1,\dots,s_j$ such that $2\le s_i\le \omega$ for all $i$, and $s_1+\dots+s_j=m+j$. We prove this by induction on $j$. The first step of induction is $j=q$. If $\omega-1$ divides $m$, then $m=(\omega-1)q$. Thus $m+q=\omega q$ and hence $m+q=s_1+\dots+s_q$, where $s_1=\cdots=s_q=w$. Now, assume that $\omega-1$ does not divide $m$. Then $m=(\omega-1)(q-1)+r$, where $0<r<\omega-1$, and so $m+q=(q-1)\omega+r+1=s_1+\dots+s_q$, where $s_i=\omega$ for $1\le i\le q-1$ and $s_q=r+1$ with $2\leq s_q<\omega$. So the claim is proved for $j=q$.

Now, let $j$ be an integer with $q<j\le m$, and assume inductively that there exist integers $s_1,\dots,s_{j-1}$, with $2\le s_i\le \omega$, such that $\sum_{i=1}^{j-1}s_i=m+(j-1)$. Since $j\le m$, there exist $1\le i\le j-1$ such that $s_i>2$. Otherwise, $\sum_{i=1}^{j-1}s_i=2(j-1)=m+(j-1)$. This implies that $j\leq m=j-1$, which is a contradiction. Therefore, we may assume that $s_{j-1}>2$. Then
$$
m+j=\sum_{i=1}^{j-1}s_i+1=\sum_{i=1}^{j-2}s_i+(s_{j-1}-1)+2=\sum_{i=1}^j s'_i,$$
where $s'_i=s_i$ for $1\le i\le j-2$, $s'_{j-1}=s_{j-1}-1$ and $s'_j=2$. We have $2\le s'_i\le\omega$ for all $1\le i\le j$. So the claim is proved. 

Next, let $q\leq j\leq m$ be an integer. We present a monomial $f$ of degree $m+j$ and show that it is a minimal generator of $I(G)^{(m)}$. Let $2\leq s_1\leq s_2\leq\cdots\le s_j\leq \omega$ be integers such that $m+j=\sum_{i=1}^j s_i$. Let $V(G)=\{x_1,\ldots,x_n\}$.
Consider a minimal monomial generator $u\in K_{s_j}(G)$. Without loss of generality, we may assume that $u=\prod_{i=1}^{s_j}x_i$. Then $\{x_1,\ldots,x_{s_j}\}$ forms a clique in $G$. Let $u_{\ell}=\prod_{i=1}^{s_{\ell}}x_i$ for $1\leq \ell\leq j$. Then by (a),
$$
f=u_1u_2\cdots u_j\in K_{s_1}(G)K_{s_2}(G)\cdots K_{s_j}(G)\subseteq I(G)^{(m)}.
$$
We show that $f$ is a minimal generator of 
$I(G)^{(m)}$. Suppose that this is not the case. Then there exists a minimal monomial generator $f'$ of $I(G)^{(m)}$ such that $f'$ divides $f$ and $\deg(f')=m+j'$ for some integer $j'<j$. We have $f'\in K_{s'_1}(G)K_{s'_2}(G)\cdots K_{s'_{j'}}(G)$ for integers $2\leq s'_1,\ldots,s'_{j'}\leq \omega$ with $m+j'=\sum_{i=1}^{j'} s'_i$.

Let $f=x_1^{a_1}\cdots x_n^{a_n}$ and $f'=x_1^{b_1}\cdots x_n^{b_n}$. We have $b_i\le a_i$ for all $i$. Moreover, $b_i\le j'$ for all $i$ since $f'$ is the product of $j'$ squarefree monomials. Furthermore,
\begin{equation}\label{eq:f}
    f=(\prod_{i=1}^{s_1}x_i^j)(\prod_{i=s_1+1}^{s_2}x_i^{j-1})\cdots(\prod_{i=s_{j-2}+1}^{s_{j-1}}x_i^2)(\prod_{i=s_{j-1}+1}^{s_j}x_i),
\end{equation}
with the convention that if $s_h=s_{h+1}$ for some $h$, then $\prod_{i=s_h+1}^{s_{h+1}}x_i^{j-h}=1$. 
From equation (\ref{eq:f}) we see that $a_i=0$ for $i>s_j$ and $a_i=p$, for any $s_{j-p}+1\leq i\leq s_{j-p+1}$, where $1\leq p\leq j$ and $s_0=0$. We can write $f'=gh$ where $g=\prod_{i=1}^{s_{j-j'+1}}x_i^{b_i}$ and $h=\prod_{i=s_{j-j'+1}+1}^{n}x_i^{b_i}$. Then
\begin{equation}\label{eq:j'1}
\sum_{i=1}^{j'} s'_i\ =\ \deg(f')=\deg(g)+\deg(h)\ \le\ j's_{j-j'+1}+\sum_{i=s_{j-j'+1}+1}^{n}b_i.
\end{equation}
Moreover, since  $a_i=0$ for $i>s_j$ and $a_i=p$, for any $s_{j-p}+1\leq i\leq s_{j-p+1}$, where $1\leq p\leq j$, we have
\begin{equation}\label{eq:j'2}
\begin{aligned}
\sum_{i=s_{j-j'+1}+1}^{n}b_i\ &\le\ \sum_{i=s_{j-j'+1}+1}^{n}a_i\\
&\le\ (j'-1)(s_{j-j'+2}-s_{j-j'+1})+\cdots+3(s_{j-2}-s_{j-3})\\
&\phantom{aaaa}+2(s_{j-1}-s_{j-2})+(s_j-s_{j-1})\\
&=\ \sum_{\ell=1}^{j'-1}\ell(s_{j-\ell+1}-s_{j-\ell}). 
\end{aligned}
\end{equation}
We have 
\begin{equation}\label{eq:j'3}
\sum_{\ell=1}^{j'-1}\ell(s_{j-\ell+1}-s_{j-\ell})=\sum_{\ell=0}^{j'-2}(\ell+1)s_{j-\ell}-\sum_{\ell=1}^{j'-1}\ell s_{j-\ell}=\sum_{\ell=0}^{j'-2}s_{j-\ell}-(j'-1)s_{j-j'+1}.
\end{equation}
From equations (\ref{eq:j'1}), (\ref{eq:j'2}) and (\ref{eq:j'3}), we see that
$$m+j'=\sum_{i=1}^{j'} s'_i\ \leq\ j's_{j-j'+1}+
\sum_{\ell=0}^{j'-2}s_{j-\ell}-(j'-1)s_{j-j'+1}=s_j+s_{j-1}+\cdots+s_{j-j'+1}.$$ 
Since $\sum_{i=1}^{j} s_i=m+j$, we conclude that $\sum_{i=1}^{j-j'} s_i\leq m+j-(m+j')=j-j'$. Since $s_i>0$ for all $i$, this implies that $s_i=1$ for all $1\leq i\leq j-j'$, which is a contradiction. 

\medskip
`Only if': Assume that $I(G)^{(m)}$ has a minimal monomial generator of degree $d$. By (a), we have $d=m+j$ for some positive integer $j\leq m$. Moreover, there exist integers $s_1,\dots,s_j$ such that $2\le s_i\le \omega$ for all $i$, and $s_1+\dots+s_j=m+j$. Then $m+j\leq j\omega$.
Thus $m\leq j(\omega-1)$, which implies that 
$\lceil m/(\omega-1)\rceil\le j$.
\end{proof}

As a corollary of Theorem~\ref{Thm:perfect}, we obtain an inequality in Conjecture~\ref{ConjA} for the family of cochordal graphs.

\begin{Corollary}\label{inequalityreg}
Let $G$ be a perfect graph. Then $\reg\, I(G)^{(m)} \ge 2m$ for all $m\ge1$. In particular, if $G$ is cochordal, then $\reg\, I(G)^{(m)} \ge \reg\, I(G)^m$.
\end{Corollary}

\begin{proof}
By Theorem~\ref{Thm:perfect}(b), we have $\beta_{0,2m}(I(G)^{(m)})\neq 0$, which proves the first statement.
Noting that any cochordal graph is a perfect graph, the second statement follows from the first statement and ~\cite[Theorem 3.2]{HHZ}, where it is shown that for a cochordal graph $G$, the ideal $I(G)^m$ has linear resolution. 
\end{proof}

In the following example, for the given graph $G$, we present some minimal monomial generators of the $6^{\textup{th}}$ symbolic power of $I(G)$ in each degree $m+j$, where $\lceil m/(\omega-1)\rceil\le j\le m$.

\begin{Example}
\rm Let $G$ be the graph depicted below. 
	
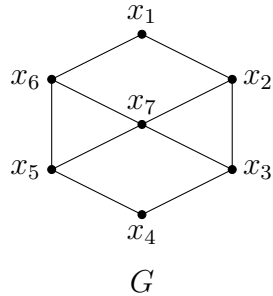
\begin{center}
	\begin{tikzpicture}[scale=0.3]\label{Pic1}
		\draw[-] (6.,6.)--(10.,4.)--(10.,0.)--(6.,-2.);
		\draw[-] (6.,6.)--(2.,4.)--(2.,0.)--(6.,-2.) ;
		\draw[-] (6.,2.)--(10.,4.);
		\draw[-] (6.,2.)--(2.,4.);
		\draw[-] (6.,2.)--(2.,0); 
		\draw[-] (6.,2.)--(10.,0.);
		\filldraw[black] (6,-2) circle (5pt) node[below]{{$x_4$}};
	 \filldraw[black] (6,6) circle (5pt) node[above]{{$x_1$}};
		\filldraw[black] (6,2) circle (5pt) node[above]{{$x_7$}};
		\filldraw (10,4) circle (5pt) node[right]{{$x_2$}};
		\filldraw (10,0) circle (5pt) node[right]{{$x_3$}};
		\filldraw (2,0) circle (5pt) node[left]{{$x_5$}};
		\filldraw (2,4) circle (5pt) node[left]{{$x_6$}};
		\filldraw (6,-4) node[below]{$G$};
 	\end{tikzpicture}
 	\end{center}
    Since $G$ and $G^c$ have no induced odd cycles of length $r>4$, it follows that $G$ is a perfect graph.  Note that $\omega=\omega(G)=3$. Consider the ideal $I=I(G)^{(6)}$. Then it follows from Theorem~\ref{Thm:perfect} that $\beta_{0,d}(I)\neq 0$ if and only if $d=6+j$, where $3=\lceil m/(\omega-1)\rceil\leq j\leq m=6$. Hence, the minimal generators of $I$ appear in degrees $9,10,11$ and $12$. Let $d=9$. As the proof of Theorem~\ref{Thm:perfect} suggests, we may write $9=3+3+3$, and choose a clique of size $3$ in $G$, say $\{x_2,x_3,x_7\}$. Hence, $u=(x_2x_3x_7)^3$ is a minimal generator of $I$ of degree 9. Similarly, $10=2+2+3+3$, and $(x_2x_3)^2(x_2x_3x_7)^2$ is a minimal generator of degree $10$. The monomials $(x_2x_3)^4(x_2x_3x_7)$ and $(x_2x_3)^6$ are minimal generators of $I$ of degrees $11$ and $12$.
\end{Example}

For a homogeneous ideal $I\subset S$, let $\alpha(I)$ denote the \textit{initial degree} of $I$, that is the minimum integer $d$ such that $I_d\neq 0$. The {\em Waldschmidt constant} of $I$ is then defined to be $\widehat{\alpha}(I)=\lim_{m\to \infty} \alpha(I^{(m)})/m$. The following corollary of Theorem \ref{Thm:perfect} recovers~\cite[Theorem 6.7(i)]{BC}.

\begin{Corollary}
Let $G$ be a perfect graph with the clique number $\omega=\omega(G)\geq 2$, and let $I=I(G)$. Then, $\widehat{\alpha}(I)=\omega/(\omega-1)$.
\end{Corollary}

\begin{proof}
By Theorem~\ref{Thm:perfect}, $\alpha(I(G)^{(m)})=m+\lceil m/(\omega-1)\rceil=\lceil m\omega/(\omega-1)\rceil$. Thus 

$$\widehat{\alpha}(I)=\lim_{m\to \infty} \frac{\lceil m\omega/(\omega-1)\rceil}{m}=\frac{\omega}{\omega-1}.$$
\end{proof}

\section{Componentwise linearity of symbolic powers of edge ideals}\label{sec2}

This section focuses on resolving Conjecture \ref{ConjB} for several classes of graphs. To set the stage for the main result, we provide a brief review of some key concepts.

A vertex $x$ of a graph $G$ is called a {\em simplicial vertex} of $G$ if its (open) neighborhood $N_G(x)=\{y\in V(G): \{x,y\}\in E(G)\}$ is a clique of $G$. A {\em perfect elimination ordering} of $G$ is an ordering $x_1>\cdots>x_n$ on the vertex set of $G$ such that $x_i$ is a simplicial vertex of the induced subgraph $G_i=G[\{x_i,x_{i+1},\ldots,x_n\}]$ for all $i$. By a classical result due to Dirac~\cite{Dirac}, $G$ is a chordal graph, if and only if, $G$ has a perfect elimination ordering.   

A graph $G$ is called a \textit{proper interval graph}, if there exists a labeling $\{x_1,\ldots,x_n\}$ on the vertex set of $G$ such that for any $i<j<k$,  $\{x_i,x_k\}\in E(G)$ implies that $\{x_i,x_j\}\in E(G)$ and $\{x_j,x_k\}\in E(G)$. It can be seen that $x_1>\ldots>x_n$ is a perfect elimination ordering of $G$. Hence, any proper interval graph is chordal.

A {\em cut vertex} of a connected graph $G$ is a vertex $x\in V(G)$ such that $G\setminus x$ is not connected.
A \textit{block} of a graph $G$ is a maximal induced subgraph $B$ of $G$ with the property that it is connected and contains no cut vertex. A graph $G$ is called a \textit{block graph} if all of its blocks are cliques of $G$. For instance, any forest is a block graph. In the sequel we use the following characterization of block graphs, for which we refer to~\cite{Ho}. 
\begin{Proposition}
\label{rem:block}
    A graph $G$ is a block graph if and only if $G$ is chordal and any two maximal cliques of $G$ have at most one vertex in common.    
\end{Proposition}

Here is a typical example of a block graph.\smallskip

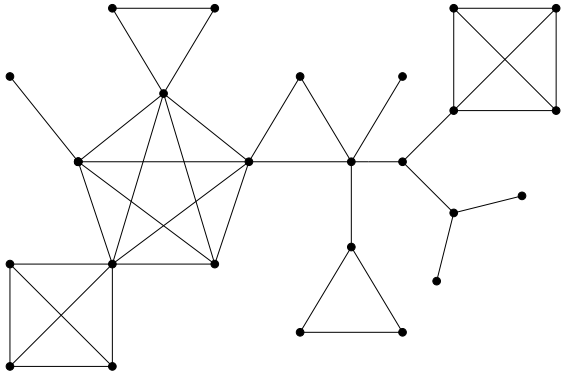
\begin{figure}[H]
\centering
\resizebox{0.5\textwidth}{!}{%
\begin{circuitikz}
\tikzstyle{every node}=[font=\LARGE]
\node at (7.75,8.5) [circ] {};
\node at (7.75,8.5) [circ] {};
\draw [short] (7.75,8.5) -- (9,9.5);
\draw [short] (9,9.5) -- (10.25,8.5);
\draw [short] (7.75,8.5) -- (8.25,7);
\draw [short] (8.25,7) -- (9.75,7);
\draw [short] (9.75,7) -- (10.25,8.5);
\node at (9,9.5) [circ] {};
\node at (8.25,7) [circ] {};
\node at (9.75,7) [circ] {};
\node at (10.25,8.5) [circ] {};
\draw [short] (9,9.5) -- (9.75,7);
\draw [short] (9,9.5) -- (8.25,7);
\draw [short] (8.25,7) -- (10.25,8.5);
\draw [short] (10.25,8.5) -- (7.75,8.5);
\draw [short] (7.75,8.5) -- (9.75,7);
\draw [short] (9,9.5) -- (8.25,10.75);
\draw [short] (8.25,10.75) -- (9.75,10.75);
\draw [short] (9.75,10.75) -- (9,9.5);
\node at (8.25,10.75) [circ] {};
\node at (9.75,10.75) [circ] {};
\node at (7.75,8.5) [circ] {};
\draw [short] (8.25,7) -- (6.75,7);
\draw [short] (6.75,7) -- (6.75,5.5);
\draw [short] (6.75,5.5) -- (8.25,5.5);
\draw [short] (8.25,7) -- (8.25,5.5);
\node at (6.75,7) [circ] {};
\node at (6.75,5.5) [circ] {};
\node at (8.25,5.5) [circ] {};
\node at (6.75,9.75) [circ] {};
\draw [short] (6.75,9.75) -- (7.75,8.5);
\draw [short] (6.75,7) -- (8.25,5.5);
\draw [short] (8.25,7) -- (6.75,5.5);
\draw [short] (10.25,8.5) -- (12,8.5);
\draw [short] (10.25,8.5) -- (11,9.75);
\draw [short] (11,9.75) -- (11.75,8.5);
\draw [short] (11.75,8.5) -- (12.5,9.75);
\draw [short] (11.75,8.5) -- (11.75,7.25);
\draw [short] (11.75,7.25) -- (11,6);
\draw [short] (11,6) -- (12.5,6);
\draw [short] (12.5,6) -- (11.75,7.25);
\draw [short] (12.5,8.5) -- (12,8.5);
\draw [short] (12.5,8.5) -- (13.25,9.25);
\draw [short] (13.25,9.25) -- (13.25,10.75);
\draw [short] (13.25,10.75) -- (14.75,10.75);
\draw [short] (14.75,10.75) -- (14.75,9.25);
\draw [short] (14.75,9.25) -- (13.25,9.25);
\node at (11.75,7.25) [circ] {};
\node at (11,6) [circ] {};
\node at (12.5,6) [circ] {};
\node at (11.75,8.5) [circ] {};
\node at (11,9.75) [circ] {};
\node at (12.5,9.75) [circ] {};
\node at (13.25,9.25) [circ] {};
\node at (12.5,8.5) [circ] {};
\node at (13.25,10.75) [circ] {};
\node at (14.75,10.75) [circ] {};
\node at (14.75,9.25) [circ] {};
\draw [short] (13.25,10.75) -- (14.75,9.25);
\draw [short] (14.75,10.75) -- (13.25,9.25);
\draw [short] (13.25,7.75) -- (13,6.75);
\draw [short] (13.25,7.75) -- (14.25,8);
\draw [short] (12.5,8.5) -- (13.25,7.75);
\node at (13.25,7.75) [circ] {};
\node at (13,6.75) [circ] {};
\node at (14.25,8) [circ] {};
\end{circuitikz}
}
\end{figure}

For an integer $n$, we set $[n]=\{1,2,\dots,n\}$. Given a non-empty subset $A$ of $[n]$ and a monomial $u=x_1^{a_1}\cdots x_n^{a_n}\in S$, we set $P_A=(x_i:i\in A)$ and $u_A=\prod_{i\in A}x_i^{a_i}$.

The following simple lemma will be used several times.
\begin{Lemma}\label{Lem:sigma}
    Let $A_1,\dots,A_m$ be non-empty subsets of $[n]$,  $k_1,\dots,k_m$ be positive integers and let $u=x_1^{a_1}\cdots x_n^{a_n}\in S$ be a monomial of degree $d$. Then, $u\in\bigcap_{i=1}^mP_{[n]\setminus A_i}^{k_i}$ if and only if
    $$
    \deg(u_{A_i})\le d-k_i,\,\,\,\,\,\textit{for all}\,\,1\le i\le m.
    $$
\end{Lemma}
\begin{proof}
    We have $u\in P_{[n]\setminus A_i}^{k_i}$ if and only if
    $
    \sum_{j\in[n]\setminus A_i}a_j\ge k_i.
    $
    Since $\sum_{j\in[n]\setminus A_i}a_j=d-\sum_{j\in A_i}a_j$, the previous inequality holds if and only if
    $
    \deg(u_{A_i})=\sum_{j\in A_i}a_j\le d- k_i.
    $
\end{proof}

Recall that an \textit{independent set} of $G$ is a subset $A$ of $V(G)$ such that no two vertices of $A$ are adjacent in $G$. The set of all independent sets of $G$ is a simplicial complex $\Delta_G$, called the \textit{independence complex} of $G$. 
A {\em vertex cover} of $G$ is a subset $C\subseteq V(G)$ which intersects each edge of $G$. A minimal set with  such property is called a {\em minimal vertex cover} of $G$.

As customary, if $\Delta$ is a simplicial complex, we denote by $\mathcal{F}(\Delta)$ the set consisting of the facets of $\Delta$. Then  
   $$
    I(G)^{(k)}\ =\ \bigcap_{A\in\mathcal{F}(\Delta_G)}P_{[n]\setminus A}^k.
    $$
    
Indeed, $P_C\in\Ass\,I(G)$ if and only if $C$ is a minimal vertex cover of $G$, which means that $C=[n]\setminus A$ for a maximal independent set $A\in\mathcal{F}(\Delta_G)$. Notice that any maximal independent set $A\in\mathcal{F}(\Delta_G)$ is a maximal clique of $G^c$. We will use this basic fact several times.

A homogeneous ideal $I\subset S$ is called {\em componentwise linear}, if the ideal $$I_{\langle d\rangle}=(f\in I, \, f \textrm{ is homogeneous of degree } d)$$ has linear resolution for any positive integer $d$. A useful approach to show that an ideal is componentwise linear is to show that it has (componentwise) linear quotients. Indeed, ideals with (componentwise) linear quotients are componentwise linear \cite[Theorem 8.2.15]{HHBook}.  Recall that a monomial ideal $I$ has {\em linear quotients} if the minimal generators of $I$ can be ordered as $u_1,\ldots,u_s$ such that for each $i = 2,\ldots,s$, the ideal $(u_1,\ldots,u_{i-1}):(u_i)$ is generated by variables.
In the following, for two monomials $u$ and $v$, we set $u:v=u/\gcd(u,v)$. Notice that
$(u_1,\ldots,u_{i-1}):(u_i)=(u_j:u_i\, |\, 1\le j\le i-1)$.

\begin{Theorem}\label{Thm:forestComLinear}
Let $G$ be one of the following graphs:
\begin{enumerate}
\item[(a)] Complement of a block graph,
\item[(b)] Complement of a proper interval graph,
\item[(c)] $G$ is a cochordal graph with the property that any vertex in $G$ belongs to at most two maximal independent sets of $G$.
\end{enumerate}
Then $\bigcap_{A\in\mathcal{F}(\Delta_G)}P_{[n]\setminus A}^{k_A}
$
is componentwise linear, for any positive integers $k_A$. In particular, $I(G)^{(k)}$ is componentwise linear for all $k$.
\end{Theorem}
\begin{proof}
Let $J=\bigcap_{A\in\mathcal{F}(\Delta_G)}P_{[n]\setminus A}^{k_A}$. We show that $J_{\langle d\rangle}$ has linear quotients for all $d$, which will imply that  $J_{\langle d\rangle}$ has linear resolution. If $J_{\langle d\rangle}=0$, there is nothing to prove. So we assume that $J_{\langle d\rangle}\neq 0$. Under any of the assumptions (a), (b) or (c), the graph $G$ is cochordal. Let $x_1>\cdots >x_n$ be a perfect elimination order of $G^c$, and consider the lex order $>$ on $J_{\langle d\rangle}$ induced by this order. Let $u=x_1^{a_1}\cdots x_n^{a_n}$ and $v=x_1^{b_1}\cdots x_n^{b_n}$ be two generators of $J_{\langle d\rangle}$ of degree $d$, with $u>v$. Let $i$ be the integer with $a_i>b_i$ and $a_j=b_j$ for $j<i$. By assumption $x_i$ is a simplicial vertex of the graph $H_i=G^c[\{x_i,x_{i+1},\ldots,x_n\}]$. Notice that since $u$ and $v$ have the same degree, there exists an integer $\ell>i$ such that $b_{\ell}>0$. 
We set 
$$L=\{x_\ell\in N_{H_i}(x_i)\ :\ b_\ell>0\}.$$

Clearly, $L\subseteq \{x_{i+1},\ldots,x_n\}$. First assume that $L=\emptyset$. Let $t$ be an integer $>i$ such that $b_t>0$. We set $w=x_iv/x_t$. Obviously, $w>v$, $w:v=x_i$ and $x_i$ divides $u:v$. Since $\deg w=d$, it remains to show that $w\in J$. By Lemma~\ref{Lem:sigma}, $\deg(u_A)\leq d-k_A$ and $\deg(v_A)\leq d-k_A$, for any $A\in\mathcal{F}(\Delta_G)$. Using Lemma~\ref{Lem:sigma} once again, we need to show that $\deg(w_A)\leq d-k_A$ for any $A\in\mathcal{F}(\Delta_G)$. Let $A\in\mathcal{F}(\Delta_G)$. If $x_i\notin A$, then $\deg(w_A)\leq \deg(v_A)\leq d-k_A$. Suppose now that $x_i\in A$. Since $A$ is a maximal clique of $G^c$, it follows that $A\cap \{x_{i+1},\ldots,x_n\}\subseteq N_{H_i}(x_i)$. Let $w=x_1^{c_1}\cdots x_n^{c_n}$. Then $c_j=b_j$ for $j<i$, and $c_i=b_i+1\leq a_i$.
Moreover, from the inclusion $A\cap \{x_{i+1},\ldots,x_n\}\subseteq N_{H_i}(x_i)$ and that $L=\emptyset$, it follows that $c_j=0$ for any $j>i$ with $x_j\in A$. Therefore, $\deg(w_A)\le\deg(u_A)\le d-k_A$, as desired. 

Now, suppose that $L$ is non-empty. Let $t$ be the minimal integer such that $x_t\in L$. We set $w=x_iv/x_t$. Write $w=x_1^{c_1}\cdots x_n^{c_n}$. We have
\begin{equation}\label{eq:abc}
   c_j\ =\ \begin{cases}
    \,b_i+1&\textup{if}\ j=i,\\
    \,b_t-1&\textup{if}\ j=t,\\
    \hfil b_j&\textup{otherwise}.
\end{cases} 
\end{equation}

We need to show that $\deg(w_A)\leq d-k_A$ for any $A\in\mathcal{F}(\Delta_G)$. Let $A\in\mathcal{F}(\Delta_G)$. 
If $x_i\notin A$, then as before, $\deg(w_A)\leq \deg(v_A)\leq d-k_A$. Now, suppose that $x_i\in A$. Next we discuss each of the cases (a), (b) and (c).

\medskip
(a) Since $\emptyset\neq L\subseteq N_{H_i}(x_i)$, we have $|N_{H_i}[x_i]|\geq 2$, where $N_{H_i}[x_i]=N_{H_i}(x_i)\cup\{x_i\}$.
By Proposition~\ref{rem:block}, for any other maximal clique $B\neq A$ of $G^c$ which contains $x_i$ we have $A\cap B=\{x_i\}$. Therefore, $N_{H_i}[x_i]$ is contained in precisely one maximal clique $C$ of $G^c$. Let $C'=C\setminus N_{H_i}(x_i)$ and let $B\ne C$ be an arbitrary maximal clique of $G^c$ containing $x_i$. We claim that $C',B\subseteq\{x_1,\dots,x_i\}$. Indeed, if $x_j\in C'$ for some $j>i$, then $x_j\in N_{H_i}(x_i)$, which is impossible. Similarly, if $x_j\in B$ for some $j>i$, then $x_j\in N_{H_i}(x_i)$ and $\{x_i,x_j\}\in B\cap C$ which is not possible by Proposition~\ref{rem:block}.

Now, we show that $\deg(w_A)\leq d-k_A$. First assume that $A=C$. Notice that $x_t\in N_{H_i}(x_i)\subseteq C$. Then by (\ref{eq:abc}) we have $\deg(w_C)=\deg(v_C)\leq d-k_C$.
Otherwise, if $A\ne C$, then as was shown above, $A\subseteq\{x_1,\dots,x_i\}$. Since $c_j=b_j=a_j$ for $j<i$ and $c_i=b_i+1\le a_i$, by (\ref{eq:abc}) we have $\deg(w_A)\le\deg(u_A)\le d-k_A$.

\medskip
(b) If $x_t\notin A$, then $x_j\notin A$ for any $j>t$, since $G^c$ is a proper interval graph and $i<t<j$. In other words, $A\subseteq\{x_j:i-s\le j\le t-1\}$ for some non-negative integer $s$. 
We have $c_{\ell}= a_{\ell}$ for any $\ell<i$, and $c_i=b_i+1\leq a_i$. Moreover, by the choice of $t$ it follows that for any $x_{\ell}\in A$ with $\ell>i$, $c_{\ell}=b_{\ell}=0$.
So $\deg(w_A)\le\deg(u_A)\le d-k_A$. If $x_t\in A$, then $\deg(w_A)=(b_i+1)+(b_t-1)+\deg(w_{A\setminus \{x_i,x_t\}})=\deg(v_A)\le d-k_A$.

\medskip
(c) By our assumption $x_i$ belongs to at most two maximal cliques of $G^c$. If $A$ is the only maximal clique which contains $x_i$, then $L\subseteq N_{H_i}[x_i]\subseteq A$. Therefore, $x_t\in A$ and hence $$\deg(w_{A})=(b_i+1)+(b_t-1)+\deg(w_{A\setminus \{x_i,x_t\}})=\deg(v_{A})\le d-k_{A}.$$ 

Now, suppose $x_i$ belongs to two maximal cliques of $G^c$, say $A_1$ and $A_2$. We may assume that $N_{H_i}[x_i]\subseteq A_1$. Then $\deg(w_{A_1})\leq \deg(v_{A_1})\le d-k_{A_1}$ because $t\in A_1$. If $A_2\cap N_{H_i}(x_i)=\emptyset$, then $\deg(w_{A_2})\leq \deg(u_{A_2})\le d-k_{A_2}$ because $A_2\subseteq\{x_1,\dots,x_i\}$. Now, suppose that $A_2\cap N_{H_i}(x_i)\neq\emptyset$. If $b_{\ell}=0$ for all $x_{\ell}\in A_2\cap N_{H_i}(x_i)$, then $\deg(w_{A_2})\leq \deg(u_{A_2})\le d-k_{A_2}$. Therefore, in this case $\deg(w_{A})\leq d-k_{A}$  for $A\in\{A_1,A_2\}$.

Finally, if $b_{\ell}>0$ for some $x_{\ell}\in A_2\cap N_{H_i}(x_i)$, we redefine $w$ as $w=x_i(v/x_t)$, where $t=\ell$. It follows that $\deg(w_{A})\le d-k_{A}$ for $A\in\{A_1,A_2\}$, as desired.
\end{proof}

The following picture gives an example of a chordal graph $G$ whose complement $G^c$ satisfies condition (c) in Theorem \ref{Thm:forestComLinear} but is neither a block graph nor a proper interval graph.
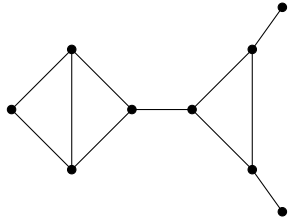
\begin{figure}[H]
\centering
\begin{tikzpicture}[scale=0.8]
\filldraw (0,0) circle (2pt) node[below]{};
\filldraw (2,0) circle (2pt) node[below]{};
\filldraw (3,0) circle (2pt) node[below]{};
\filldraw (1,1) circle (2pt) node[above]{};
\filldraw (1,-1) circle (2pt) node[below]{};
\filldraw (4,1) circle (2pt) node[above]{};
\filldraw (4,-1) circle (2pt) node[below]{};
\filldraw (4.5,-1.7) circle (2pt) node[below]{};
\filldraw (4.5,1.7) circle (2pt) node[above]{};
\draw[-] (0,0) -- (1,1) -- (2,0) -- (3,0) -- (4,1) -- (4,-1) -- (3,0);
\draw[-] (0,0) -- (1,-1) -- (2,0);
\draw[-] (4,1) -- (4.5,1.7);
\draw[-] (4,-1) -- (4.5,-1.7);
\draw[-] (1,1) -- (1,-1);
\end{tikzpicture}
\end{figure}

\begin{Corollary}\label{cor:goodcoffee}
    Let $G$ be one of the graphs considered in Theorem \ref{Thm:forestComLinear}. Then, $\reg\,I(G)^{(k)}=\reg\,I(G)^{k}=2k$ for all $k$.
\end{Corollary}

\begin{proof}
Since $I(G)^{(k)}$ is componentwise linear, by \cite[Corollary 8.2.14]{HHBook}, $\reg\,I(G)^{(k)}$ is equal to the highest degree of a minimal generator of $I(G)^{(k)}$. By Theorem~\ref{Thm:perfect}, this degree is $2k$. The equality $\reg\,I(G)^{k}=2k$ holds by~\cite[Theorem 3.2]{HHZ}. 
\end{proof}

Let $G_1$ and $G_2$ be graphs on disjoint vertex sets. The {\em join} of $G_1$ and $G_2$ is defined as the graph $G_1*G_2$ with the vertex set $V(G_1)\cup V(G_2)$ and the edge set $E(G_1)\cup E(G_2)\cup\{\{x,y\}:\, x\in V(G_1), y\in V(G_2)\}$. One can define the join operation inductively for any finite number of graphs on disjoint vertex sets.

\begin{Corollary}
Let $G=G_1*\cdots*G_r$, where each $G_i$ belongs to one of the families of graphs in Theorem~\ref{Thm:forestComLinear}. Then $\reg\,I(G)^{(k)}=\reg\,I(G)^{k}=2k$ for all $k$.
\end{Corollary}

\begin{proof}
By~\cite[Theorem 3.2]{KKS} and Corollary~\ref{cor:goodcoffee}, 
$$\reg\,I(G)^{(k)}=\max\{\reg\,I(G_j)^{(i)}-i+k\ :\  1\leq i\leq k, 1\le j\le r\}=2k.$$ 
On the other hand, the graph $G^c$ is the disjoint union of $G_1^c,\ldots,G_r^c$. Therefore it is chordal. Hence, \cite[Theorem 3.2]{HHZ} implies that $\reg\,I(G)^{k}=2k$.
\end{proof}

\section{The second symbolic power of edge ideals}\label{sec3}

In this section, we address Conjectures \ref{ConjB} and \ref{ConjC} for the second symbolic power of edge ideals. We prove that $I(G)^{(2)}$ has linear quotients if $G$ is a cochordal graph.

To this end, let $G$ be a cochordal graph. Since $G$ is perfect, Theorem \ref{Thm:perfect} implies that
\begin{align}
    \label{eq:Minh2}I(G)^{(2)}\ &=\ K_3(G)+I(G)^2.
\end{align}

We will need the following technical lemmas. For a monomial ideal $I$, the unique set of minimal monomial generators of $I$
is denoted by $\mathcal{G}(I)$.

\begin{Lemma}\label{Lem:LinQuotTech}
Let $I_1,\dots,I_t$ be equigenerated monomial ideals with linear quotients, with $\alpha(I_j)=d_j$ and $d_1\le\cdots\le d_t$. Suppose that the following property is satisfied:\medskip\\
$(*)$ For all $u\in\mathcal{G}(I_i)$ and $v\in\mathcal{G}(I_j)$ with $i<j$ and $\deg(u:v)>1$, there exists a monomial $w$ belonging to $$(I_1\cup\dots\cup I_{j-1})\cup\{w\in\mathcal{G}(I_j):w>v\ \text{in the linear quotients order of}\ I_j\}$$ such that $\deg(w:v)=1$ and $w:v$ divides $u:v$.\medskip\\
Then $I=I_1+\dots+I_t$ has linear quotients. 
\end{Lemma}
\begin{proof}
    We prove the statement by induction on $t\ge1$. For $t=1$, there is nothing to prove. Let $t\ge2$. To simplify the notation, we set $J=I_1+\dots+I_{t-1}$, $L=I_t$ and $I=J+L$. By induction, $J$ has linear quotients. If $L\subseteq J$, then $I=J$ and there is nothing to prove. Suppose now that $L\not\subseteq J$. Let $u_1,\dots,u_m$ and $v_1,\dots,v_\ell$ be linear quotients orders of $J$ and $L$, respectively. Since $J$ is generated in degrees $d_1,\dots,d_{t-1}$ and $L$ is generated in degree $d_t$, it follows that $\mathcal{G}(I)=\mathcal{G}(J)\cup\{v_{j_1},\dots,v_{j_k}\}$, for a certain $k\ge1$ and $1\le j_1<\dots<j_k\le\ell$. We claim that $u_1,\dots,u_m,v_{j_1},\dots,v_{j_k}$ is a linear quotients order of $I$. Since $u_1,\dots,u_m$ is already a linear quotients order, it is enough to show that
    $$
    (u_1,\dots,u_m,v_{j_1},\dots,v_{j_{i-1}}):v_{j_i}
    $$
    is generated by variables for all $i$. For later use, let $H=(u_1,\dots,u_m,v_1,\dots,v_{j_{i}-1})$.

    Consider a generator $v_{j_r}:v_{j_i}$. Then, there exists $s<j_i$ such that $v_{s}:v_{j_i}=x_p$ and $x_p$ divides $v_{j_r}:v_{j_i}$. If $s=j_q$ for some $q$, then we are done. Otherwise, $v_{s}$ is not a minimal generator and there exists $u_h$ which divides $v_{s}$. Then $u_h:v_{j_i}$ is not one and it divides $v_{s}:v_{j_i}=x_p$. Hence $u_h:v_{j_i}=x_p$ and we are done.

    Consider now a generator $u_r:v_{j_i}$. If $\deg(u_r:v_{j_i})=1$, we are done. Otherwise, if $\deg(u_r:v_{j_i})>1$, the property $(*)$ implies that there exists $w\in H$ such that $w:v_{j_i}=x_p$, and $x_p$ divides $u_r:v_{j_i}$. Since $H=(u_1,\dots,u_m,v_{j_1},\dots,v_{j_{i-1}})$, there exists $w'\in\{u_1,\dots,u_m,v_{j_1},\dots,v_{j_{i-1}}\}$ such that $w'$ divides $w$. Then $w':v_{j_i}$ divides $w:v_{j_i}=x_p$. Since $w',v_{j_i}\in\mathcal{G}(I)$, it follows that $w':v_{j_i}=x_p$, and this concludes the proof.
\end{proof}
\begin{Lemma}\label{Lem:Somayeh2}
    Let $x\in X$ be a variable, $I_1\subset S=K[X]$ and $I_2\subseteq K[X\setminus\{x\}]$ be monomial ideals with linear quotients such that $I_2\subseteq I_1$. Suppose that $\mathcal{G}(xI_1)\subseteq\mathcal{G}(I)$. Then $I=xI_1+I_2$ has again linear quotients.
\end{Lemma}
\begin{proof}
    Let $u_1,\dots,u_m$ and $v_1,\dots,v_\ell$ be linear quotients order of $I_1$ and $I_2$, respectively. If $I=xI_1$ there is nothing to prove. Otherwise, $\mathcal{G}(I)=\mathcal{G}(xI_1)\cup\mathcal{G}(I_2)$, since $I_2\subseteq K[X\setminus\{x\}]$. We claim that $xu_1,\dots,xu_m,v_1,\dots,v_\ell$ is a linear quotients order of $I$. To this end, since $xu_1,\dots,xu_m$ is a linear quotients order, it is enough to show that $(xu_1,\dots,xu_m,v_1,\dots,v_{i-1}):v_i$ is generated by variables for all $i$.
    Consider a generator $v_{r}:v_{i}$. Then, there exists $s<i$ such that $v_{s}:v_{i}=x_p$ and $x_p$ divides $v_{r}:v_{i}$ and we are done.  
Now, consider a generator $xu_r:v_{i}$. From $I_2\subseteq K[X\setminus\{x\}]$, we know that  the variable $x$ divides $xu_r:v_{i}$. Since $v_{i}\in I_2\subseteq I_1$, there exists $w\in\mathcal{G}(I_1)$ which divides $v_{i}$. Thus $xw\in\mathcal{G}(I)$, $xw:v_{i}=x$ and we are done.
\end{proof}

Another lemma which is required is the following

\begin{Lemma}\label{Lem:PJ}
    Let $I=I(G)$ be an edge ideal with linear quotients and let $P\subset S$ be a monomial prime ideal. Then $PI$ has linear quotients.
\end{Lemma}
\begin{proof}
	Up to a relabeling, we may assume $P=(x_1,\dots,x_t)$. Let $u_1,\dots,u_m$ be a linear quotients order of $I$. We proceed by induction on $m$. If $m=1$, $x_1u_1,\dots,x_tu_1$ is a linear quotients order of $PI$. Let $m>1$ and $L=(u_1,\dots,u_{m-1})$. Then $I=(L,u_m)$, $L$ is again an edge ideal with linear quotients and so by induction $PL$ has a linear quotients order, say, $v_1,\dots,v_h$. Let $x_{j_1}u_{m},\dots,x_{j_s}u_m$, with $1\le j_1<\dots<j_s\le t$, be the monomials in $\mathcal{G}(PI)\setminus\mathcal{G}(PL)$. We claim that 
\begin{equation}\label{standardlq}
 v_1,\dots,v_h,x_{j_1}u_m,\dots,x_{j_s}u_m   
\end{equation}
    is a linear quotients order of $PI$. Since by induction, $v_1,\dots,v_h$ is a linear quotients order of $PL$, it remains to show that $(v_1,\dots,v_h,x_{j_1}u_m,\dots,x_{j_{i-1}}u_m):x_{j_i}u_m$ is generated by variables for all $1\le i\le s$. It is clear that $x_{j_p}u_m:x_{j_i}u_m=x_{j_p}$ is a variable for all $1\le p<i$. Consider now the monomial $v_\ell:x_{j_i}u_m$. Then $v_\ell=x_pu_q$ for some $1\le p\le t$ and some $1\le q<m$. If $\deg(v_\ell:x_{j_i}u_m)=1$, there is nothing to prove. Suppose that $\deg(v_\ell:x_{j_i}u_m)\geq 2$. Let $u_q=x_rx_s$.
    Then at least one of the variables $x_r$ and $x_s$ divides $v_\ell:x_{j_i}u_m$, say  $x_r$. Consider $u_q:u_m$. Since $I$ has linear quotients, there exists $k<m$ such that $u_k:u_m$ is a variable that divides $u_q:u_m$, and so $u_k:u_m$ divides $x_rx_s$. If $u_k:u_m=x_r$, then $x_{j_i}u_k:x_{j_i}u_m=x_r$. Notice that $x_{j_i}u_k\in \mathcal{G}(PL)$. So in this case we are done. Now, assume that $u_k:u_m=x_s$. If $x_s$ divides $v_\ell:x_{j_i}u_m$, then the same argument as before can be applied. Now, suppose that $x_s$ does not divide $v_\ell:x_{j_i}u_m$. Then $v_\ell:x_{j_i}u_m=x_px_r$. Since $u_k:u_m=x_s$, $x_s$ does not divide $u_m$. These imply that $j_i=s$ and hence $u_k$ divides $x_{j_i}u_m$. Therefore, $x_pu_k:x_{j_i}u_m=x_p$ divides 
    $v_\ell:x_{j_i}u_m$ and $x_pu_k\in \mathcal{G}(PL)$.    
\end{proof}

The following remark will be needed in the proof of Theorem \ref{Thm:I2}.
\begin{Remark}\label{rem:PJ}
    \rm  Let $I=I(G)$ be an edge ideal with linear quotients and let $P=(x_{j_1},\dots,x_{j_t})\subset S$ be a monomial prime ideal. Let $u_1,\dots,u_m$ be a linear quotients order of $I$. Then $\mathcal{G}(PI)=\{v_1,\dots,v_h\}\subseteq\{x_{j_p}u_q:1\le p\le t,1\le q\le m\}$. Consider the following order of monomials
\begin{equation}\label{eq:minRep}
    x_{j_1}u_1>\cdots>x_{j_t}u_1> x_{j_1}u_2>\cdots>x_{j_t}u_2>\cdots>x_{j_1}u_m>\cdots>x_{j_t}u_m.
\end{equation}
Notice that each $v_\ell$ is equal to at least one monomial $x_{j_p}u_q$ in the above list. We call 
$x_{j_p}u_q=v_\ell$ the \textit{standard presentation} of $v_\ell$ if $x_{j_p}u_q$ is the biggest monomial equal to $v_\ell$ in the order (\ref{eq:minRep}). Then, the order (\ref{eq:minRep}) induces a total order $>$ on $\mathcal{G}(PI)$ defined for any $v_\ell,v_s\in\mathcal{G}(PI)$ by setting $v_\ell>v_s$ if the standard presentation of $v_\ell$ is bigger than the standard presentation of $v_s$ in the order (\ref{eq:minRep}). It follows from the proof of Lemma \ref{Lem:PJ} that $PI$ has linear quotients with respect to the order $>$.
Indeed, in the ordering (\ref{standardlq}), one may assume by induction that $v_1,\ldots,v_h$ is the desired order. Since $x_{j_1}u_m,\dots,x_{j_s}u_m$ belong to $\mathcal{G}(PI)\setminus\mathcal{G}(PL)$, it follows that they are standard presentations, and so (\ref{standardlq}) is the desired linear quotients order of $PI$.
\end{Remark}

The following result strengthens a result by Minh \textit{et al.} \cite[Theorem 3.3]{MNPTV} in the case that $G$ is a cochordal graph.
\begin{Theorem}\label{Thm:I2}
    Let $G$ be a cochordal graph. Then $I(G)^{(2)}$ has linear quotients. In particular, $\reg\,I(G)^{(2)}=\reg\,I(G)^2=4$.
\end{Theorem}
\begin{proof}
   Let $x_1>\dots>x_n$ be a perfect elimination order of $G^c$. We prove the theorem by induction on $n\ge2$. If $n=2$, then $V(G)=\{x_1,x_2\}$, and $I(G)=(0)$ or $I(G)=(x_1x_2)$ and $I(G)^{(2)}=I(G)^2=(x_1^2x_2^2)$ has linear quotients.

   Suppose now $n>2$. Let $G_1=G\setminus\{x_1\}$ and $G_2=G[N_G(x_1)]$. Then, by the proof of~\cite[Theorem 3.2]{Mor} we have  
   \begin{align}
       \label{eq:I2}I(G)^2\ &=\ (x_1K_1(G_2))^2+x_1K_1(G_2)I(G_1)+I(G_1)^2,\\
       \label{eq:K_3}K_3(G)\ &=\ x_1I(G_2)+K_3(G_1),
   \end{align}
   with $I(G_1)\subseteq K_1(G_2)$, $K_3(G_1)\subseteq I(G_2)$, and these four ideals appearing in the inclusion relations have linear quotients.
   
   We set $P=K_1(G_2)$, and note that $P$ is a monomial prime ideal. Then, by (\ref{eq:Minh2}), (\ref{eq:I2}) and (\ref{eq:K_3}),
   \begin{equation}\label{eq:induction}
        I(G)^{(2)}\ =\ x_1[I(G_2)+x_1P^2+PI(G_1)]+I(G_1)^{(2)}.
    \end{equation}

   Let $\mathcal{G}(P)=V(G_2)=\{x_{j_1},\dots,x_{j_t}\}$. We may assume that $1\le j_1<\dots<j_t\le n$. Then the linear quotient orders of $PI(G_1)$ are determined as in the Remark \ref{rem:PJ}. On the set $\mathcal{G}(x_1P^2)$ we fix the lex order induced by $x_1>x_2>\cdots>x_n$. Obviously, this is a linear quotients order of $x_1P^2$. 
   
   Set $I_1=I(G_2)$, $I_2=x_1P^2$, $I_3=PI(G_1)$, and $L=I_1+I_2+I_3$. Since $I_1,I_2,I_3$ are equigenerated with linear quotients (see Lemma~\ref{Lem:PJ}), by Lemma \ref{Lem:LinQuotTech}, it is enough to show that $L$ satisfies the property $(*)$. For this aim, let $u\in\mathcal{G}(I_h)$ and $v\in\mathcal{G}(I_\ell)$, with $h<\ell$, such that $\deg(u:v)>1$.
   
   Suppose $h=1$. Hence $u=x_ix_j\in I(G_2)$ with $x_i>x_j$ and $x_i,x_j\in P$. We have $\ell=2$ or $\ell=3$. 
   
   Suppose $\ell=2$. Then $v=x_1(x_px_q)\in x_1P^2$ with $x_p\ge x_q$. Since $\deg(u:v)>1$, we have $u:v=u=x_ix_j$ and $p\ne i,j$. Note that $x_p,x_q,x_i,x_j\in\mathcal{G}(P)=V(G_2)$. If $x_i>x_p$, then $w=x_1(x_ix_q)\in I_2$, and $w>v$ in the linear quotients order of $I_2$. Moreover, $w:v=x_i$ divides $u:v$, as wanted. Otherwise, suppose $x_p>x_i$.  We claim that $x_px_i\in I(G_2)$ or $x_px_j\in I(G_2)$. Suppose this is not the case, then $\{x_p,x_i\},\{x_p,x_j\}\in E(G^c)$. Since $x_p$ is a simplicial vertex of $G^c[x_p,x_{p+1},\ldots,x_n]$ and $x_p>x_i>x_j$, it would follow that $\{x_i,x_j\}\in E(G^c)$, which is absurd. Therefore, $x_px_i\in I(G_2)$ or $x_px_j\in I(G_2)$. If, for instance, $w=x_px_i\in I(G_2)$, then $w:v=x_i$ divides $u:v$ and the property $(*)$ is again satisfied. Otherwise, if $w=x_px_j\in I(G_2)$, once again $w:v=x_j$ divides $u:v$ and the property $(*)$ is satisfied.
   
   Suppose $\ell=3$. Then $v=x_p(x_rx_s)\in PI(G_1)$ with $x_p\in P$ and $x_rx_s\in\mathcal{G}(I(G_1))$. We assume that $v=x_p(x_rx_s)$ is the standard presentation of $v$. Since $\deg(u:v)>1$, then $u:v=u$ and so $p\ne i,j$. If $x_i>x_p$, then $w=x_i(x_rx_s)\in\mathcal{G}(PI(G_1))$, and $w>v$ in the linear quotients order of $PI(G_1)$ by Remark \ref{rem:PJ}. Then $w:v=x_i$ divides $u:v$, and the property $(*)$ is verified in such a case. Suppose now $x_p>x_i>x_j$. As shown before, $x_px_i\in I(G_2)$ or $x_px_j\in I(G_2)$. If $w=x_px_i\in I(G_2)$, then $w:v=x_i$ divides $u:v$ and the property $(*)$ is satisfied. We proceed similarly if $x_px_j\in I(G_2)$.

   Suppose now $h=2$. Then $\ell=3$. In this case $x_1$ divides $u:v$, and $v=x_p(x_ix_j)$ with $x_p\in P$ and $x_ix_j\in I(G_1)$. Since $I(G_1)\subseteq P$, we may assume that $x_i\in P$. Then $w=x_1(x_px_i)\in\mathcal{G}(x_1P^2)$, $w:v=x_1$ divides $u:v$, and the property $(*)$ is satisfied.

   Hence, $L$ has linear quotients. Notice that $I(G)^{(2)}=x_1L+I(G_1)^{(2)}$. By induction, $I(G_1)^{(2)}$ has linear quotients. Since $I(G_1)\subseteq P$ and $K_3(G_1)\subseteq I(G_2)$, equation (\ref{eq:Minh2}) implies that $I(G_1)^{(2)}\subseteq L$. We claim that $\mathcal{G}(x_1L)\subseteq\mathcal{G}(I(G)^{(2)})$. Then, Lemma \ref{Lem:Somayeh2} implies that $I(G)^{(2)}$ has linear quotients, as desired.

   Suppose that $\mathcal{G}(x_1L)\setminus\mathcal{G}(I(G)^{(2)})\ne\emptyset$. Then, there exist monomials $u\in\mathcal{G}(x_1L)$ and $v\in\mathcal{G}(I(G_1)^{(2)})$ such that $v$ divides $u$  properly. Since $x_1L,I(G_1)^{(2)}$ are generated in degrees three and four, we have $\deg(u)=4$  and $\deg(v)=3$. Equation (\ref{eq:induction}) implies that $u\in x_1^2P^2+x_1PI(G_1)$. If $u=x_1^2(x_px_q)$ with $x_px_q\in P^2$, since $v\in K[x_2,\dots,x_n]$, then $v$ should divide $x_px_q$, which is not possible. Otherwise, if $u\in x_1PI(G_1)$, then $v$ should divide $u/x_1=x_p(x_ix_j)$, where $x_p\in P$ and $x_ix_j\in I(G_1)$. Since $\deg(v)=3$, from the equation $I(G_1)^{(2)}=K_3(G_1)+I(G_1)^2$ we have $u/x_1=v\in K_3(G_1)$. Thus $u=x_1v\in x_1K_3(G_1)\subseteq x_1I(G_2)$, against the fact that $u\in\mathcal{G}(x_1L)$. Hence $\mathcal{G}(x_1L)\subseteq\mathcal{G}(I(G)^{(2)})$, and this concludes the proof.
\end{proof}

\noindent\textbf{Acknowledgment.}
A. Ficarra was partly supported by INDAM (Istituto Nazionale di Alta Matematica), and also by the Grant JDC2023-051705-I funded by
MICIU/AEI/10.13039/501100011033 and by the FSE+. S. Moradi is supported by the Alexander von Humboldt Foundation.


\begin{thebibliography}{99}

 

\bibitem{AHH}  A. Aramova, J. Herzog, T. Hibi, \textit{Ideals with stable Betti numbers}, Adv. Math. {\bf 152} (2000), no. 1, 72–77.

\bibitem{BQ1} S. Bandari, A. A. Qureshi,
{\em Ideals with linear quotients and componentwise polymatroidal ideals},
Mediterranean Journal of Mathematics, {\bf 20} (2023), no. 2, Paper No. 53, 15 pp.

\bibitem{BQ2}
S. Bandari, A. A. Qureshi,
{\em Componentwise linear ideals and exchange properties},
(2024), preprint \url{https://arxiv.org/abs/2405.20645}.

\bibitem{BC} C. Bocci, S. Cooper, E. Guardo, B. Harbourne, M. Janssen, U. Nagel, A. Seceleanu, A. Van Tuyl, T. Vu, \textit{The Waldschmidt constant for squarefree monomial ideals}, J. Algebraic Combin. {\bf 44} (2016), no. 4, 875--904.

\bibitem{CHT}
S. D. Cutkosky, J. Herzog and N. V. Trung,
{\em Asymptotic behaviour of the Castelnuovo-Mumford regularity},
Compositio Math. {\bf 118} (1999), no. 3, 243--261.

\bibitem{Dirac} G. A. Dirac,
{\em  On rigid circuit graphs},
Abh. Math. Sem. Univ. Hamburg {\bf 38} (1961),  71-76.

\bibitem{DHNT}
L. X. Dung, T. T. Hien, H. D. Nguyen, and T. N. Trung,
{\em Regularity and Koszul property of symbolic powers of monomial ideals}, Math. Z. {\bf 298} (2021), no. 3-4, 1487--1522.

\bibitem{Fic0}
A. Ficarra,
{\em Shellability of componentwise discrete polymatroids},
(2023), preprint \url{https://arxiv.org/abs/2312.13006}.

\bibitem{FVT} C. A. Francisco, A. Van Tuyl, \textit{Some families of componentwise linear monomial ideals}, Nagoya Math. J.  {\bf 187} (2007), 115-156.

\bibitem{Froberg88} R. Fr\"oberg, \textit{On Stanley-Reisner rings}, Topics in algebra, Part 2 (Warsaw, 1988), 57--70, Banach Center Publ., 26, Part 2, PWN, Warsaw, 1990.

\bibitem{GHOS}
Y. Gu, H. T. Ha, J. L. O'Rourke, and J. W. Skelton,
{\em Symbolic powers of edge ideals of graphs}, Commun. Algebra {\bf 48} (2020), 3743–3760.

\bibitem{HH99} J. Herzog, T. Hibi, \textit{Componentwise linear ideals}, Nagoya Math. J. {\bf 153} (1999), 141--153.

\bibitem{HHBook} J.~Herzog, T.~Hibi, \emph{Monomial ideals}, Graduate texts in Mathematics {\bf 260}, Springer, 2011.

\bibitem{HHT} J. Herzog, T. Hibi, N. Trung, \textit{Symbolic powers of monomial ideals and vertex cover algebras}, Adv. Math. {\bf 210} (2007), 304--322.

\bibitem{HHZ} J.~Herzog, T.~Hibi, X.~Zheng, {\em Monomial ideals whose powers have a linear resolution}, Math. Scand. {\bf 95} (2004), no. 1, 23--32.

\bibitem{Ho} E. Howorka, \textit{On metric properties of certain clique graphs}, J. Combin. Theory Ser. B {\bf 27} (1979), no. 1, 67--74.

\bibitem{JK}
A. V. Jayanthan, R. Kumar,
{\em Regularity of Symbolic Powers of Edge Ideals}, 
J. Pure Appl. Algebra {\bf 224} (2020), 106306, 12 pp.

\bibitem{K} 
V. Kodiyalam, 
{\em Asymptotic behaviour of Castelnuovo-Mumford regularity},
 Proc. Amer. Math. Soc. {\bf 128} (2000), 407--411.

\bibitem{KKS} A. Kumar, R. Kumar, and R. Sarkar,
{\em Certain algebraic invariants of edge ideals of join of graphs},
J. Algebra Appl. {\bf 20} (2021), no. 6, Paper No. 2150099, 12 pp.


\bibitem{MNPTV} N. C. Minh, L. D. Nam, T. D. Phong, P. T. Thuy, and T. Vu,
{\em Comparison between regularity of small powers of symbolic powers and ordinary powers of an edge ideal},
J. Combin. Theory Ser. A {\bf 190} (2022), 105621.

\bibitem{MTr} 
N. C. Minh, T. N. Trung, {\em Regularity of symbolic powers and arboricity of matroids}, 
Forum Mathematicum {\bf 31} (2019), Issue 2, 465--477.

\bibitem{MVu} 
N. C. Minh, T. Vu,  Survey on regularity of symbolic powers of an edge ideal. Commutative algebra, 569--588, Springer, Cham, 2021.

\bibitem{MM} F. Mohammadi, S. Moradi, \textit{Weakly polymatroidal ideals with applications to vertex cover ideals},
Osaka Journal of Mathematics {\bf 47}(3) (2010), 627–636.

\bibitem{Mor} S. Moradi, {\em t-clique ideal and t-independence ideal of a graph}, Comm. Algebra {\bf 46} (2018), no. 8, 3377--3387.

\bibitem{F1}
S. A. Seyed Fakhari, 
{\em Regularity of symbolic powers of edge ideals of unicyclic graphs}, J. Algebra {\bf 541} (2020), 345--358.

\bibitem{F3}
S. A. Seyed Fakhari, 
{\em Regularity of symbolic powers of edge ideals of Cameron-Walker graphs}, Comm. Algebra {\bf 48}, Issue 12, (2020), 5215–5223.

\bibitem{F2}
S. A. Seyed Fakhari, 
{\em Regularity of symbolic powers of edge ideals of chordal graphs}, Kyoto J. Math. {\bf62} (2022), no. 4, 753--762.



\bibitem{F4}
S. A. Seyed Fakhari, 
{\em On the regularity of small symbolic powers of edge ideals of graphs}, Math. Scand. {\bf 129} (2023), no. 1, 39--59. 

\bibitem{FY}
S. A. Seyed Fakhari, S. Yassemi,
{\em Improved bounds for the regularity of edge ideals of graphs}, Collectanea Math. {\bf 69} (2018), 249--262.

\bibitem{SVV} A. Simis, W. V. Vasconcelos,  R. H. Villarreal, \textit{On the ideal theory
of graphs}, J. Algebra {\bf 167} (1994), 389–416.

\bibitem{SJZ}
A. Soleyman Jahan, X. Zheng,
{\em Ideals with linear quotients},
Journal of Combinatorial Theory,
Series A {\bf 117} (2010), no. 1, 104--110.

\bibitem{Vil} R. H. Villarreal,
{\em Rees algebras and polyhedral cones of ideals of vertex covers of perfect graphs},
J. Algebraic Combin. {\bf 27} (2008), 293–305.

\end{thebibliography}
\end{document}